\documentclass[reqno,twoside,11pt,english]{amsart}
\usepackage{amsmath,amsfonts,amssymb,amsthm,epsfig}
\newcommand{\RNum}[1]{\uppercase\expandafter{\romannumeral #1\relax}}
\usepackage{comment}

\usepackage{color}
\usepackage[final,colorlinks,linkcolor=blue,anchorcolor=red,citecolor=blue]{hyperref}
\usepackage{graphicx}
\usepackage{titletoc,epsf}
\usepackage{amsmath,amsfonts,latexsym,amsthm,amsxtra,amssymb,bbm}
\allowdisplaybreaks
\usepackage{graphicx,float}
\usepackage{tikz}
\usetikzlibrary{intersections,patterns}
\usetikzlibrary{calc,spy}
\usepackage{ifthen}
\usepackage{float}

\allowdisplaybreaks

\newtheorem{thm}{Theorem}[section]
\newtheorem{lem}[thm]{Lemma}
\newtheorem{op}[thm]{Open Problem}
\newtheorem{cor}[thm]{Corollary}
\newtheorem{prop}[thm]{Proposition}
\newtheorem{assertion}{Assertion}[section]

\newtheorem{step}{Step}[section]

\newtheorem{cl}{Claim}[section]

\newtheorem{ca}{Case}%[section]
\newtheorem{sca}[section]{Subcase}
\newtheorem{scl}[section]{Subclaim}
\newtheorem{conj}[equation]{Conjecture}

\theoremstyle{definition}
\newtheorem{defn}[thm]{Definition}

\newtheorem{ques}[equation]{Question}
\newtheorem{rem}[thm]{Remark}
\newtheorem{exam}[thm]{Example}

\newcounter {own}
\def\theown {\thesection       .\arabic{own}}
\numberwithin{equation}{section}

\newenvironment{pf}[1][]{%
	\vskip 3mm
	\noindent
	\ifthenelse{\equal{#1}{}}%
	{{\slshape Proof. }}%
	{{\slshape #1.} }%
}%
{\qed\bigskip}

\newcommand{\IR}{{\mathbb R}}

\newcommand{\diam}{{\operatorname{diam}}}

\newcommand{\Card}{{\operatorname{Card}}}

\newcommand{\bas}{\begin{assertion}}
	\newcommand{\eas}{\end{assertion}}
\newcommand{\ben}{\begin{enumerate}}
	\newcommand{\een}{\end{enumerate}}
\newcommand{\bst}{\begin{step}}
	\newcommand{\est}{\end{step}}

\def\be{\begin{equation}}
	\def\ee{\end{equation}}

\newcommand{\bee}{\begin{enumerate}}
	\newcommand{\eee}{\end{enumerate}}

\newcommand{\blem}{\begin{lem}}
	\newcommand{\elem}{\end{lem}}
\newcommand{\bthm}{\begin{thm}}
	\newcommand{\ethm}{\end{thm}}
\newcommand{\bcor}{\begin{cor}}
	\newcommand{\ecor}{\end{cor}}
\newcommand{\beg}{\begin{exam}}
	\newcommand{\eeg}{\end{exam}}
\newcommand{\begs}{\begin{examples}}
	\newcommand{\eegs}{\end{examples}}
\newcommand{\bdefe}{\begin{defn}}
	\newcommand{\edefe}{\end{defn}}
\newcommand{\bprob}{\begin{prob}}
	\newcommand{\eprob}{\end{prob}}
\newcommand{\bques}{\begin{ques}}
	\newcommand{\eques}{\end{ques}}
\newcommand{\bei}{\begin{itemize}}
	\newcommand{\eei}{\end{itemize}}
\newcommand{\bcon}{\begin{conj}}
	\newcommand{\econ}{\end{conj}}
\newcommand{\bop}{\begin{op}}
	\newcommand{\eop}{\end{op}}

\newcommand{\bstep}{\begin{step}}
	\newcommand{\estep}{\end{step}}

\newcommand{\bca}{\begin{ca}}
	\newcommand{\eca}{\end{ca}}
\newcommand{\bsca}{\begin{sca}}
	\newcommand{\esca}{\end{sca}}

\newcommand{\bcl}{\begin{cl}}
	\newcommand{\ecl}{\end{cl}}

\newcommand{\bscl}{\begin{scl}}
	\newcommand{\escl}{\end{scl}}

\newcommand{\bcons}{\begin{conjs}}
	\newcommand{\econs}{\end{conjs}}
\newcommand{\bprop}{\begin{prop}}
	\newcommand{\eprop}{\end{prop}}
\newcommand{\br}{\begin{rem}}
	\newcommand{\er}{\end{rem}}
\newcommand{\brs}{\begin{rems}}
	\newcommand{\ers}{\end{rems}}
\newcommand{\bo}{\begin{obser}}
	\newcommand{\eo}{\end{obser}}
\newcommand{\bos}{\begin{obsers}}
	\newcommand{\eos}{\end{obsers}}
\newcommand{\bpf}{\begin{pf}}
	\newcommand{\epf}{\end{pf}}
\newcommand{\ba}{\begin{array}}
	\newcommand{\ea}{\end{array}}
\newcommand{\beq}{\begin{eqnarray}}
	\newcommand{\beqq}{\begin{eqnarray*}}
		\newcommand{\eeq}{\end{eqnarray}}
	\newcommand{\eeqq}{\end{eqnarray*}}

\newcommand{\ds}{\displaystyle}

\newcounter{minutes}
\divide\time by 60
\newcounter{hours}
\multiply\time by 60 \addtocounter{minutes}{-\time}
\begin{document}

	\title[Characterizations of $\varphi$-inner uniform domains and $\varphi$-uniform domains]{Characterizations of $\varphi$-inner uniform domains and $\varphi$-uniform domains in metric spaces}
	
	\author[S.-J. Gao,  M. He ]{Shu-Jing Gao and Miao He}
	
	\address[S.-J. Gao]{School of Mathematics, Shandong University, Jinan and Research Center for Mathematics and Interdisciplinary Sciences, Shandong University, Qingdao, P. R. China}
	\email{gsjing9@163.com}

	\address[M. He]{Research Center for Mathematics and Interdisciplinary Sciences, Shandong University, Qingdao, P. R. China}
	\email{hm293428@163.com}

	\date{}
	%\dedicatory{Dedicated to Professor~Pekka Koskela on the occasion of his 65th birthday}
	%\thanks{$^*$Corresponding author: Manzi Huang}
	%${}^{\mathbf{*}}$ Corresponding author}

\begin{abstract}
In this paper, we study $\varphi$-uniform/-inner uniform domains in  metric spaces.  In our first main result, we establish quantitative characterizations of three different definitions of $\varphi$-uniformity/-inner uniformity in length spaces. In our second main result, we obtain a geometric sufficient condition of $\varphi$-uniform/-inner uniform domains in doubling length spaces, which extends the corresponding result of Guo-Huang-Wang [arXiv 2025] and Gao-Guo-Huang-Wang [arXiv 2025] for linear case. 
%The proof relies on a deep theorem about geometric characterization of Gromov hyperbolicity, established recently in the seminal work of Guo-Huang-Wang \cite{GHW}.   	
\end{abstract}

%\date{\today}

\thanks{Both authors are supported by the Young Scientist Program of the Ministry of Science and Technology of China (No.~2021YFA1002200), the NSF of China (12311530037) and the Taishan Scholar Project. }
%M. Huang and X. Wang are partly supported by NSF of China (No.12371071, No.12271189 and No.12571081). }

\maketitle

%%%%%%%%%%%%%%%%%%%%%%%%%%%%%%%%%%%%%%%%%%%
\section{Introduction}\label{sec-1}
%%%%%%%%%%%%%%%%%%%%%%%%%%%%%%%%%%%%%%%%%%%

The class of uniform domains was first introduced by Martio and Sarvas \cite{MS} in their study of injectivity theorems of analytic functions and suitable higher dimensional extension. The idea for its definition was motivated by the earlier celebrated work of F. John on elasticity \cite{Jo}. Due to its nice analytic/geometric properties and intimate relation with quasiconformal maps, it plays a central role in modern geometric function theory and  higher dimensional domain analysis; see for instance  \cite{GO,GePa,GGGKN-2024,Guo-2015,Hajlasz-Koskela-2000,HeKo,Herron-Koskela-1991,HLPW,Jo81,Martin-1985,Ma,NV-1991,Vai-1988,Vai7}.

The recent studies on generalized quasidisks (see e.g.~\cite{GKT,Guo-2015-PAMS}) motivates people to consider nonlinear extensions of the class of uniform/inner uniform domains.  Let $\varphi$ be a continuous, increasing function with $\varphi(0) = 0$ and $ \varphi(t) \geq t$ for all $t > 0$.  Following \cite{Guo-Koskela-2014} and \cite{Guo-2015}, we recall the definition of $\varphi$-uniform/-inner uniform domains below. 
\begin{defn}[$\varphi$-length uniform domain]\label{def:varphi uniform domain}
	A domain $\Omega$ in a metric space $X=(X,d)$ is called  {\it $\varphi$-length uniform}, if there is a constant $c\geq 1$ and each pair of points $z_{1},z_{2}$ in $\Omega$ can be joined by a rectifiable curve $\gamma$ in $\Omega$ satisfying
	\begin{enumerate}
		\item\label{con1} $\ds\min_{j=1,2}\{\ell_d (\gamma [z_j, z])\}\leq \varphi(c\, d_\Omega(z))$ for all $z\in \gamma$, and
		\item\label{con2} $\ell_d(\gamma)\leq \varphi(c\,d(z_{1},z_{2}))$,
	\end{enumerate}
	where $\ell_d(\gamma)$ denotes the arc-length of $\gamma$ with respect to the metric $d$,
	$\gamma[z_{j},z]$ the subcurve of $\gamma$ between $z_{j}$ and $z$, and $d_\Omega(z):=d(z,\partial \Omega)$ denotes the distance from $z$ to the boundary of $\Omega$. In a $\varphi$-length uniform domain $\Omega$, any curve $\gamma\subset \Omega$, which satisfies conditions (1) and (2) above, is called a {\it $\varphi$-length uniform curve}.
\end{defn}

When $\ell_d$ in \eqref{con1} is replaced by distance or diameter and (\ref{con2}) is replaced by diameter,  then the corresponding domain is called {\it $\varphi$-dist uniform domain} or {\it $\varphi$-diam uniform domain} and the corresponding curve is called a {\it $\varphi$-dist uniform curve} or {\it $\varphi$-diam uniform curve}, respectively. 

If the condition \eqref{con2} in Definition \ref{def:varphi uniform domain} is replaced by the weaker inequality
\begin{equation*}\label{eq:def for inner quasiconvexity}
	\ell_d(\gamma)\leq \varphi\left(c\,\sigma_\Omega( z_{1}, z_{2})\right),
\end{equation*}
where $\sigma_\Omega$ is the {\it inner distance} in $\Omega$ defined by
$$\sigma_\Omega(z_1,z_2)=\inf \{\ell(\alpha):\; \alpha\subset \Omega\;
\mbox{is a rectifiable curve joining}\; z_1\; \mbox{and}\; z_2 \},$$
then $\Omega$ is said to be $\varphi$-{\it length inner uniform} and the corresponding curve $\gamma$ is called a \emph{$\varphi$-length inner uniform curve}. When the context is clear, we often drop the subscript $\Omega$ from $\sigma_{\Omega}$ and simply write $\sigma$. One can similarly define the class of $\varphi$-dist inner uniform domains and $\varphi$-diam inner uniform domains.
In the linear case $\varphi(t)=t$, $\varphi$-uniform/-inner uniform domains are simply uniform/inner uniform domains.

\begin{rem}\label{rmk:on tech assumptions on varphi}

To avoid complicated formulations for our results, we assume throughout this paper that $ \varphi^{-1} $ is doubling and that $ t \mapsto t/\varphi^{-1}(t) $ is decreasing when $0<t<1$ and increasing when $t\geq1$. Here doubling requires that $ \varphi^{-1}(2t) \leq C\varphi^{-1}(t) $ for some constant $ C $ and all $ t > 0 $. The second condition is in practise harmless, but the doubling condition rules out the case of exponential exterior cusps. 
\end{rem}

It is immediate from the definitions that 
$$\varphi\text{-length uniform domains} \Rightarrow \varphi\text{-diam uniform domains}\\\Rightarrow \varphi\text{-dist uniform domains}.$$
For the reverse direction, when $\varphi=t$ is linear, we know that a $c$-diam uniform domain $\Omega \subset \mathbb{R}^n$ is $c_1$-length uniform with $c_1 = c_1(c, n)$ (see \cite{Ma}) and a $c$-dist uniform domain $\Omega$ in a normed space is $c_2$-length uniform with $c_2 = c_2(c) \leq 64 c^3$ (see \cite{Vai2004}). 

Our first main result below shows that when $\varphi$ is nonlinear, the above quantitative characterizations of different $\varphi$-uniform domains remain valid. When $\varphi(t)=t$ is linear, a similar result was obtained in the recent work \cite{LSW-2026} via a different argument. 

\begin{thm}\label{thm:the equivalence of the three types of varphi uniform domain}
Let $X$ be a length space and $\Omega$ be a proper subdomain in $X$. Then the following statements hold. 
\begin{enumerate}
	\item\label{result 1} If $\Omega$ is $\varphi$-dist uniform, then it is $\psi$-diam uniform, where $$\psi(t)=2\varphi\Bigl(16c\,\varphi\bigl(2c\,\varphi(c\,t)\bigr)\Bigr).$$ 
	\item\label{result 2} If $X$ is a $Q$-doubling metric space and $\Omega$ is $\varphi$-diam John, then it is $\eta$-length John. %where $\eta=$. 
	Moreover, if $\Omega$ is $\varphi$-diam uniform, then it is $\eta$-length uniform. Here $$\eta(t)=\max\{2\eta_1(2t), 2\eta_2(t)\},$$
	with
	 $$\eta_1(t) = \int_0^{\varphi(Ct)} Q^{\log_2 160C\frac{s}{\varphi^{-1}(s)}} ds,\quad  \eta_2(t)=\left(Q^{\log_2 80\,\varphi\left(2Ct\right)/t}+C'\right)t,$$
	 where $C$, $C'$
	are constants depending only on the given data with $C\geq C'\geq1$, and $\eta_1$ is assumed to converge.
\end{enumerate}
\end{thm}

Recall that a metric space $X=(X,d)$ is called \emph{$Q$-doubling}, if for each ball $B(x,r)$, every $r/2$-separated subset of $B(x,r)$ has at most $Q$ points.
A simple volume comparison implies that $\IR^n$ (equipped with the standard Euclidean distance) is $2^n$-doubling. 

Following Guo and Koskela \cite{Guo-Koskela-2014}, 
%\begin{defn}[$\varphi$-John domain]\label{def:nonlinear John}
a domain $\Omega$ in a metric space $X=(X, d)$ is called {\it $\varphi$-length John}, if there is a constant $c\geq1$ such that for each $x,y \in \Omega$, there exists a rectifiable curve $\gamma\subset \Omega$ joined $x$ and $y$, satisfying Definition \ref{def:varphi uniform domain} (\ref{con1}). Any such curve $\gamma$ is called a {\it  $\varphi$-length John curve}. 
When the length on the left-hand side of  Definition \ref{def:varphi uniform domain} (\ref{con1}) is replaced by distance or diameter, then the corresponding domain is called a {\it $\varphi$-dist John domain} or a {\it $\varphi$-diam John domain}, and the corresponding curve is called a {\it $\varphi$-dist John curve} or a {\it $\varphi$-diam John curve}, respectively. 

The following corollary is an immediate consequence.
\begin{cor}\label{cor:quasihyperbolic geodesic}
	Let $X$ be a $Q$-doubling length space and $\Omega$ be a proper subdomain in $X$ such that $(\Omega,k)$ is geodesic. If $\Omega$ is a $\varphi$-diam John domain, then every quasihyperbolic geodesic is an $\eta$-length John curve. Moreover, if $\Omega$ is a $\varphi$-diam uniform, then every quasihyperbolic geodesic is an $\eta$-length uniform curve. Here $\eta$ is as in Theorem~\ref{thm:the equivalence of the three types of varphi uniform domain}.
\end{cor}

Next, we turn to the geometric sufficient condition of $\varphi$-uniform/-inner uniform domains in doubling metric spaces. For this, we first recall the definition of quasihyperbolic metric, which was initially introduced by Gehring and Palka \cite{GePa} for domains in $\IR^n$, and then has been extensively studied in \cite{GO}.
The {\it quasihyperbolic length} of a rectifiable arc
$\gamma$ 
in a proper domain $\Omega\subsetneq (X,d)$ is defined as
$$
\ell_{k}(\gamma):=\int_{\gamma}\frac{|dz|}{d_\Omega(z)}.
$$
For any $z_1$, $z_2$ in $\Omega$, the {\it quasihyperbolic distance}
$k_\Omega(z_1,z_2)$ between $z_1$ and $z_2$ is set to be
$$k_\Omega(z_1,z_2)=\inf_{\gamma}\{\ell_k(\gamma)\},
$$
where the infimum is taken over all rectifiable arcs $\gamma$
joining $z_1$ and $z_2$ in $\Omega$. An arc $\gamma$ from $z_1$ to $z_2$ is called a {\it quasihyperbolic geodesic} if
$\ell_k(\gamma)=k_\Omega(z_1,z_2)$. Clearly, each subarc of a quasihyperbolic
geodesic is a quasihyperbolic geodesic. 

Akin to the linear case \cite{GO,GePa,Guo-2015,Guo-Koskela-2014,NV-1991,ZLR}, quasihyperbolic geodesic also plays a central role in the study of geometric aspect of $\varphi$-uniform/-inner uniform domains. 
In \cite{Guo-Koskela-2014}, Guo and Koskela proved that each quasihyperbolic geodesic in a bounded simply connected planar $\varphi$-John domain is quantitatively $\varphi$-inner uniform, which in particular implies that bounded simply connected planar $\varphi$-John domains are indeed $\varphi$-inner uniform. Shortly after that, Guo \cite{Guo-2015} extended this result to bounded Gromov hyperbolic $\varphi$-John domains in $\mathbb{R}^n$. Here we will not give the precise definition of Gromov hyperbolic domains, but it is useful to notice that quasiconformal images of inner uniform domains are typical examples of Gromov hyperbolic domains; see \cite{BHK,BB4,HRWZ-2025,ZLR} for more information. 

In a recent major breakthrough, Guo-Huang-Wang \cite{GHW} proved that in doubling length spaces, Gromov hyperbolicity is completely characterization by the following geometric ball separation condition: A domain $\Omega$ in a metric space $(X,d)$ is said to satisfy the \emph{$C$-ball separation condition}, $C>0$, if for each quasihyperbolic geodesic $\gamma_{xy}\subset \Omega$, every $z\in \gamma_{xy}$ and every curve $\gamma\subset \Omega$ joining $x$ and $y$, it holds
\begin{equation}\label{eq:def for ball separation}
	B_{\sigma}\left(z,Cd_{\Omega}(z)\right)\cap \gamma\neq \emptyset.
\end{equation}
To emphasize the distance $\sigma$, we also say that $(\Omega,\sigma)$ satisfies the $C$-ball separation condition. As an immediate consequence, we know that each quasihyperbolic geodesic in a bounded $\varphi$-John domain $\Omega\subset \mathbb{R}^n$, which satisfies the ball separation condition, is $\varphi$-inner uniform, quantitatively (see \cite{GGGKN-2024}). In the linear case $\varphi(t)=t$, a similar kind of geometric characterization of uniform/inner uniform domains in doubling length spaces was obtained in \cite{GGHW}, under a  geometric local linear connectivity condition even weaker than John. 

Our second goal of this paper is to further extend these results to the class of $\varphi$-uniform/-inner uniform domains in doubling length spaces. For this, we first recall the following nonlinear local connectivity condition, originally introduced in \cite{GKT}. A domain  $\Omega\subset X$ is called \emph{$\varphi$-locally connected} (written for short as $\varphi$-LC), if the following two conditions are satisfied:
\begin{description}
	\item[$\varphi$-LC-1] each pair of points $a,b\in \Omega\cap {B(x,r)}$ can be joined in $\Omega\cap {B(x,\varphi(cr))}$;
	\item[$\varphi$-LC-2] each pair of points $a,b\in \Omega\backslash B(x,r)$ can be joined in $\Omega\backslash B(x,\varphi^{-1}(r/c))$.
\end{description}
It is called $\varphi$-LC-$i$ ($i\in \{1,2\}$) if only the $i$-th condition is satisfied. Above, the constant $c$ is required to be independent of $x$ and $r$. 

With the above notion at hand, our second main result reads as follows. 
\begin{thm}\label{thm:main result varphi inner uniform}
	Let $(X,d)$ be a $Q$-doubling length space and $\Omega$ be a proper subdomain in $X$ such that $(\Omega,k)$ is geodesic. If $\Omega$ is $\varphi$-LC-2 and satisfies the $c_0$-ball separation condition, then every 
	quasihyperbolic geodesic $\gamma$ in $\Omega$ is an $\eta$-inner uniform curve. In particular, $\Omega$ is an $\eta$-inner uniform domain. Here $\eta$ is as in Theorem \ref{thm:the equivalence of the three types of varphi uniform domain}.
\end{thm}
A similar geometric sufficient condition of $\varphi$-uniformity holds as well.    
\begin{thm}\label{thm:main result varphi uniform}
	Let $(X,d)$ be a $Q$-doubling length space and $\Omega$ be a proper subdomain in $X$ such that $(\Omega,k)$ is geodesic. If $\Omega$ is $\varphi$-LC and satisfies the $c_0$-ball separation condition, then every 
	quasihyperbolic geodesic $\gamma$ in $\Omega$ is a $\phi$-uniform curve. In particular, $\Omega$ is a $\phi$-uniform domain. 
	Here 
	$$\phi(t)=\max\left\{\eta(t), \eta_3(t)\right\},$$	
	where $\eta$ is as in Theorem \ref{thm:the equivalence of the three types of varphi uniform domain} and
 $$\eta_3(t)=e^{C_3Q^{\log_2\varphi(C_4t)/\eta^{-1}(t)}}t.$$	
with $C_3$, $C_4$ constants depending only on the given data and $\eta_3$ increasing, $\eta_3(0)=0$.
\end{thm}

Theorems \ref{thm:main result varphi inner uniform} and \ref{thm:main result varphi uniform} extend the corresponding results of \cite{GHW} and \cite{GGHW} to nonlinear uniform/inner uniform setting. For the proofs of both theorems, we shall use a deep result established in \cite[Theorem 1.9]{GHW}, which says that under the assumptions of the above theorems, $\Omega$ satisfies the so-called \emph{Gehring-Hayman inequality}. Recall that	a domain $\Omega$ in a metric space $(X,d)$ is said to satisfy the $C$-Gehring-Hayman inequality, $C > 1$, if for all $x, y \in \Omega$ and each quasihyperbolic geodesic $\gamma_{xy} \subset \Omega$ with end points $x$ and $y$, it holds
$$\ell(\gamma_{xy}) \le C\sigma_{\Omega}(x,y).$$
To emphasize the metric $d$, we also say that $(\Omega, d)$ satisfies the $C$-Gehring-Hayman inequality.

%Notice that while inner uniform domains in $Q$-doubling length spaces are Gromov hyperbolic, this is no longer true for nonlinear $\varphi$-inner uniform domains. For example, 
%\[
%\Omega = \bigl( \mathbb{R}^2 
%\setminus \{(n,0): n \in \mathbb{N}\} \bigr) 
%\setminus \bigl\{ (x,y) \mid 0 < x < 1,\ |y| < x^2 \bigr\}
%\]
%is $\varphi$-inner uniform but not Gromov hyperbolic. 

\medskip
%\textbf{Structure.} The structure of this paper is as follows. Section \ref{sec:elmentary} contains some elementary results that are necessary for later proofs. In Section \ref{sec:proofs}, we prove our main theorems. 
%%Finally, in Appendix, we proved that $\varphi$-dist (diam, length) John domain is $\varphi'$-LC-2 with $\varphi'(t)=2\varphi(2ct)$ and  $\varphi$-dist (diam, length) uniform domain is $\varphi'$-LC with $\varphi'(t)=2\varphi(2ct)$. 
%\smallskip

\textbf{Notations.} Throughout this paper, $X$ is always assumed to be a metric space. For a given metric $d$ on $X$,  $x\in X$ and $r>0$, we define 
$$B_d(x,r):=\{y\in X: d(x,y)<r\},$$ $$S_d(x,r):=\{y\in X: d(x,y)=r\}.$$
For each proper subdomain $\Omega\subset X$, we use $\Lambda_{xy}(\Omega)$ to represent the set of all quasihyperbolic geodesics in $\Omega$ with end points $x$ and $y$, and use $\gamma_{xy}$ to denote some quasihyperbolic geodesic in $\Lambda_{xy}(\Omega)$.
Meanwhile, we use  $\Gamma_{xy}(\Omega)$ to represent the set of all curves in $\Omega$ with end points $x$ and $y$.

\section{Some elementary estimates}\label{sec:elmentary}

For any $z_1$, $z_2$ in $\Omega$, let $\gamma\subset \Omega$ be an arc with end points $z_1$ and $z_2$. Then we have the following elementary estimates (see for instance \cite[Section 2]{Vai3})
\beq\label{(2.1)}
\ell_{k}(\gamma)\geq
\log\Big(1+\frac{\ell(\gamma)}{\min\{d_\Omega(z_1), d_\Omega(z_2)\}}\Big)
\eeq
and
\beq\label{(2.2)}
\begin{aligned}
	k_{\Omega}(z_1, z_2) &\geq  \log\Big(1+\frac{\sigma_\Omega(z_1,z_2)}{\min\{d_\Omega(z_1), d_\Omega(z_2)\}}\Big)
	\\
	&\geq
	\log\Big(1+\frac{|z_1-z_2|}{\min\{d_\Omega(z_1), d_\Omega(z_2)\}}\Big)
	\geq
	\Big|\log \frac{d_\Omega(z_2)}{d_\Omega(z_1)}\Big|.
\end{aligned}
\eeq

The technical assumptions on $\varphi$ imply the following result.

\begin{lem}[{\cite[Lemma 3.5]{Guo-Koskela-2014}}]\label{lem-3-1}
	Let $C_1 \geq 1$, $C_2 \geq 1$, and $C_3 \geq 1$ be given. There exists a constant $C$, depending only on $C_0, C_1, C_2$ and $C_3$, such that
	$$C_1 \varphi(C_2 t) + C_3 t \leq \varphi(C t)$$
	for all $t > 0$. Above, $C_0$ is the doubling constant of $\varphi^{-1}$.	
\end{lem}

The following estimate is useful in our later proofs.
\begin{lem}[{\cite[Lemma 2.2]{GHW}}]\label{Lemma-2.1}
	Let $(X, d)$ be a length space and $\Omega$ a proper domain of $X$. Suppose that there are a constant $a>1$ and a quasihyperbolic geodesic $\gamma$ joining $x_1, x_2$ in $\Omega$ such that $d(x_1,x_2)\leq a^{-1}d_{\Omega}(x_1)$.  Then $$k_{\Omega}(x_1,x_2)\leq \frac{10a}{9(a-1)}\frac{d(x_1,x_2)}{d_{\Omega}(x_1)}\leq \frac{10}{9}(a-1)^{-1}$$
	and $$\ell(\gamma)\leq \frac{10a}{9(a-1)}e^{\frac{10}{9}(a-1)^{-1}}d(x_1,x_2).$$
\end{lem}

We shall repeatedly use the following elementary estimate in $Q$-doubling metric spaces.
\begin{lem}[{\cite[Lemma 4.1.11]{HKST-2015}}]\label{qs-5}
	Let $X=(X,d)$ be a $Q$-doubling metric space with constant $Q$ and $\Omega\subset X$ a domain.	Fix $R>0$ and $a\geq 1$ and let $r=\frac{R}{a}$. Then for any $x\in \Omega$, the ball $B(x,R)$ contains
	at most $b$ balls with radius $r$ such that they are disjoint from each other, where $b\leq Q^{[\log_2 a]}$. Here and hereafter, $[\cdot]$ means the greatest integer part.
\end{lem}

%Guo-Huang-Wang \cite{GHW} proved that Gromov hyperbolic domains in $Q$-doubling length spacesare indeed completely characterized by the geometric ball separation condition,  so it's easy to obtain the following Lemma, which is used for the proof of Theorem \ref{thm:main result varphi uniform}.

We shall use the following result in our later proof of Theorem \ref{thm:main result varphi uniform}. It was first proved in \cite[Lemma 2.27]{GHW-1} for $X$ being Gromov hyperbolic and under the current assumption with the aid of \cite[Theorem 1.7]{GHW}. 

\begin{lem}\label{Lemma-2.5}
	Suppose that $X$ is a $Q$-doubling length space, $D\subsetneq X$ satisfies $c_0$-ball separation condition, $ x_1 $, $ x_2 $, $ x_3 $ are points in $D$, 
	%and $\mu = \mu(Q, c_0)$ is a constant. 
	%Suppose further that 
	%there is a constant  $\tau \geq 0$ such that 
	and $ k_D(x_2, x_3) \leq 1$. If  
	$$\min \{k_D(x_1, x_2), k_D(x_1, x_3)\} > 1,$$
	then for any $ \vartheta \in \Lambda_{x_1 x_2} $ and $ \gamma \in \Lambda_{x_1 x_3} $, there exists a constant $C_0 = C_0(Q, c_0)\geq4$ such that
	$$\max \{\ell(\vartheta), \ell(\gamma)\} \leq C_0\min \{\ell(\vartheta), \ell(\gamma)\}.$$
	%where $C_0=4e^{4 + \frac{22}{3}\mu}.$
\end{lem}
%\begin{proof}
%By \cite[Theorem 1.7]{GHW}, we know $D$ is $\delta$-Gromov hyperbolic with $\delta=\delta(Q, c_0)$. Then following \cite[Lemma 2.27]{GHW-1}, we can obtain the above fact.
%\end{proof}

%%%%%%%%%%%%%%%%%%%%%%%%%%%%%%%%%%%%%%%%%%%
%%%%%%%%%%%%%%%%%%%%%%%%%%%%%%%%%%%%%%%%%%%
\section{Proofs of the main results}\label{sec:proofs}
\subsection{Proof of Theorem \ref{thm:the equivalence of the three types of varphi uniform domain}}
In this subsection, we investigate the relationship between the three types of $\varphi$-uniform domain.

\begin{proof}[Proof of Theorem~\textup{\ref{thm:the equivalence of the three types of varphi uniform domain} \eqref{result 1}}]
	Assume that $\Omega$ is a $\varphi$-dist uniform domain. Fix $a,b\in \Omega$. We first consider the easy case	$\max\{d_\Omega(a), d_\Omega(b)\}>d(a,b).$
Without loss of generality, assume that $d_\Omega(a)>d(a,b)$. 
Let $$\varepsilon=\min\{d_\Omega(a)-d(a,b), d(a,b)\}$$
and $\gamma$ be an arc joining $a$ and $b$ with $\ell(\gamma)< d(a,b)+\varepsilon\leq d_\Omega(a)$.
Then $\gamma\subset B(a,d_\Omega(a))\subset\Omega$. For every $z\in\gamma$,
$$\ell(\gamma[z,b])=\ell(\gamma)-\ell(\gamma[a,z])\leq d_\Omega(a)-\ell(\gamma[a,z])\leq d_\Omega(z),$$
which implies that
\beqq
\min\{\operatorname{diam}(\gamma[a,z]),\operatorname{diam}(\gamma[z,b])\}\leq\ell(\gamma[z,b])\le d_\Omega(z)
\le \varphi(c\,d_\Omega(z)).
\eeqq
Moreover,
\beqq
\operatorname{diam}(\gamma)\le \ell(\gamma)<d(a,b)+\varepsilon\leq2d(a,b)\le \varphi(2c\,d(a,b)).
\eeqq
Thus the desired conclusion holds after a harmless enlargement of the constant.
Hence it remains to consider $\max\{d_\Omega(a), d_\Omega(b)\}\le d(a,b).$
	
Choose a $\varphi$-dist uniform curve $\gamma_0\subset \Omega$ joining $a$, $b$
and $z_0\in \gamma_0$ such that	$d(a,z_0)=d(z_0,b)=:t.$
	Consequently,
	\begin{equation}\label{eq:t-basic}
		\frac{1}{2}d(a,b)\le t\le \operatorname{diam}(\gamma_0)\le \varphi(c\,d(a,b)).
	\end{equation}
and
	\begin{equation}\label{eq:z0-boundary}
		t\le \varphi\bigl(c\,d_\Omega(z_0)\bigr).
	\end{equation}

	It suffices to construct an arc $\alpha$ joining $a$ and $z_0$ such that 
	for any $x\in\alpha[a,z_0]$, 
	\beq\label{eq:2.1}
	\diam (\alpha[a,x])\leq\psi(d_{\Omega}(x))
	\eeq
	and \beq\label{eq:2.2}
	\diam(\alpha)\leq\frac{1}{2}\psi(d(a,b)),
	\eeq
	where $$\psi(t)=2\varphi\Bigl(16c\,\varphi\bigl(2c\,\varphi(c\,t)\bigr)\Bigr).$$
	Similarly, we may obtain an analogous arc $\beta$ joining $b$ and $z_0$.
	Removing a loop from $\alpha\cup \beta$ yields the desired $\psi$-diam uniform curve.
	
	Set $$r=\frac14 d_\Omega(a)\le \frac14d(a,b) \leq \frac{t}{2},$$ then we may choose the unique integer $m\ge 2$ such that
	\begin{equation}\label{eq:dyadic-scale}
		2^{m-1}r\leq t< 2^m r.
	\end{equation}
	Choose a point $x_1\in \gamma_0[a,z_0]\cap S(a,r).$ 
	Inductively, for
	$1\le j\le m-1$, assume $x_j$ has been chosen on $S(a,2^{j-1} r)$.
	 Choose a $\varphi$-dist uniform curve	$\alpha_j$ joining $x_j$ and $z_0$, such that for every $x\in \alpha_j$,
	\begin{equation}\label{eq:alpha-j-cigar}
	\min\{	d(x_j,x), d(x,z_0)\}\leq \varphi\bigl(c\,d_\Omega(x)\bigr)
	\end{equation}
	and
	\begin{equation}\label{eq:alpha-j-turning}
		\operatorname{diam}(\alpha_j)\le \varphi\bigl(c\,d(x_j,z_0)\bigr).
	\end{equation}
	Then define $x_{j+1}$ as the first point of $\alpha_j$ lying on the
	sphere $S(a,2^j r)$.
	
	Let $\alpha_0'\subset \Omega$ be an arc joining $a$ and $x_1$ with length less than $2r$; such an arc exists since $X$ is a length space. For $1 \le j \le m-2$, set $\alpha_j' := \alpha_j[x_j, x_{j+1}]$, and let $\alpha_{m-1}' := \alpha_{m-1}$. After deleting finitely many loops from the union
	$\alpha_0' \cup \alpha_1' \cup \cdots \cup \alpha_{m-1}',$
	the resulting arc $\alpha$ joins $a$ and $z_0$. We now verify that it has the required properties.
	
	For $0 \le j \le m-2$, the inclusion $\alpha_0' \cup \cdots \cup \alpha_j' \subset \overline{B}(a, 2^j r)$ implies
	\begin{equation}\label{eq:prefix-diameter}
		\operatorname{diam}(\alpha_0' \cup \cdots \cup \alpha_j') \le 2^{j+1} r.
	\end{equation}
	For the remaining index $j = m-1$, we have
	\[
	d(x_{m-1}, z_0) \le d(x_{m-1}, a) + d(a, z_0) = 2^{m-2}r + t < 2t,
	\]
	and therefore, by \eqref{eq:alpha-j-turning},
	\[
	\operatorname{diam}(\alpha_{m-1}) \le \varphi(2ct).
	\]
	In view of \eqref{eq:t-basic} and the inequality $\varphi(2ct) \ge t \ge 2^{m-1}r$, it follows that
	\begin{equation}\label{eq:last-arc-diam}
		\operatorname{diam}(\alpha) \le 2\varphi(2ct) \le 2\varphi(2c\,\varphi(c\,d(a,b))).
	\end{equation}
	Thus \eqref{eq:2.2} is established.
	
We now prove \eqref{eq:2.1}. Fix $x \in \alpha$. By construction, there exists a unique 
$0 \le j \le m-1$ such that $x \in \alpha_j'$. We distinguish three cases according 
to the value of $j$.

\textbf{Case 1:}  $j = 0$. 

Then
\[
\operatorname{diam}(\alpha[a,x]) \le \ell(\alpha_0') \le 2r.
\]
Moreover,
\[
d_\Omega(x) \ge d_\Omega(a) - d(a,x) \ge d_\Omega(a) - \ell(\alpha_0') \ge 2r,
\]
hence
\[
\operatorname{diam}(\alpha[a,x]) \le d_\Omega(x). %\le \psi(d_\Omega(x)).
\]
%Thus \eqref{eq:2.1} holds in this case.

\textbf{Case 2:} $1 \le j \le m-2$. 

Then by \eqref{eq:prefix-diameter}, we have
\begin{equation}\label{eq:T_j}
	\operatorname{diam}(\alpha[a,x]) \leq\operatorname{diam}(\alpha_0' \cup \cdots \cup \alpha_j') \le 2^{j+1}r.
\end{equation}
Note that 
$$d(x_{j-1}, x_j) \ge d(x_j,a)-d(x_{j-1}, a) \geq2^{j-2}r$$
 and from \eqref{eq:dyadic-scale},
\[
d(x_j, z_0) \ge d(a, z_0) - d(x_j, a) \ge t - 2^{j-1}r > 2^{j-2}r.
\]
Thus, since $x_j \in \alpha_{j-1}$, applying \eqref{eq:alpha-j-cigar} to $\alpha_{j-1}$ at $x_j$ yields
	\begin{equation}\label{eq:xj-boundary}
		2^{j-2}r\le\min\{d(x_{j-1},x_j), d(x_j,z_0)\}\le \varphi\bigl(c\,d_\Omega(x_j)\bigr).
	\end{equation}
	
	We claim that 
	$$2^{j-2}r \le \varphi\bigl(2c\,\varphi(c\,d_\Omega(x))\bigr).$$
Indeed, if $d(x, x_j) \le \frac{1}{2c}\varphi^{-1}(2^{j-2}r)$, then by \eqref{eq:xj-boundary},
\[
d_\Omega(x) \ge d_\Omega(x_j) - d(x, x_j) \ge \frac{1}{2c}\varphi^{-1}(2^{j-2}r),
\]
hence
\[
2^{j-2}r \le \varphi(2c\,d_\Omega(x)) \le \varphi\bigl(2c\,\varphi(c\,d_\Omega(x))\bigr).
\]
If instead $d(x, x_j) > \frac{1}{2c}\varphi^{-1}(2^{j-2}r)$, since $x \in \alpha_j' \subset \overline{B}(a, 2^j r)$ and $j \le m-2$,
\[
d(x, z_0) \geq d(a, z_0)-d(x, a)\ge t - 2^j r \ge 2^j r \ge \frac{1}{2c}\varphi^{-1}(2^{j-2}r),
\]
so
\[
\min\{d(x, x_j), d(x, z_0)\} \ge \frac{1}{2c}\varphi^{-1}(2^{j-2}r).
\]
Applying \eqref{eq:alpha-j-cigar}, we get
\[
2^{j-2}r \le \varphi\bigl(2c\,\varphi(c\,d_\Omega(x))\bigr).
\]
Thus the claim holds.

Combining the claim with \eqref{eq:T_j} yields
\[
\operatorname{diam}(\alpha[a,x]) \le 2^{j+1}r \le 8\varphi\bigl(2c\,\varphi(c\,d_\Omega(x))\bigr).
\]
	
\textbf{Case 3:} $j = m-1$. 

Then by \eqref{eq:last-arc-diam},	\begin{equation}\label{eq:last-start}
		\operatorname{diam}(\alpha[a,x])\le\operatorname{diam}\alpha\le 2\varphi(2ct).
	\end{equation}

Claim that 
$$t \le 8\varphi\bigl(2c\,\varphi(c\,d_\Omega(x))\bigr).$$
Indeed, if $d(x, x_{m-1}) \le \frac{1}{2c}\varphi^{-1}(2^{m-3}r)$, using \eqref{eq:xj-boundary} at $j=m-1$ and arguing as above, we get
\[
2^{m-3}r \le \varphi(2c\,d_\Omega(x)),
\]
hence
\[
t \le 2^m r \le 8\varphi(2c\,d_\Omega(x)) \le 8\varphi\bigl(2c\,\varphi(c\,d_\Omega(x))\bigr).
\]
If $d(x,z_0)\le \frac{1}{2c}\,\varphi^{-1}(t),$
	then using \eqref{eq:z0-boundary},
	we have
	$$d_{\Omega}(x)\geq d_{\Omega}(z_0)-d(x,z_0)\geq\frac{1}{2c}\,\varphi^{-1}(t), $$
	hence $$t\le \varphi\bigl(2c\,d_\Omega(x)\bigr)\le 8\varphi\bigl(2c\,\varphi(c\,d_\Omega(x))\bigr).$$
If neither of the above holds, that is
	$$d(x,x_{m-1})>\frac{1}{2c}\,\varphi^{-1}(2^{m-3}r)
	\quad \text{and} \quad
	d(x,z_0)>\frac{1}{2c}\,\varphi^{-1}(t).$$
Applying \eqref{eq:alpha-j-cigar} to $\alpha_{m-1}$ at $x$, we obtain
	$$\frac{1}{2c}\,\varphi^{-1}(2^{m-3}r)\le \varphi\bigl(c\,d_\Omega(x)\bigr).$$
	Therefore
	\beqq
	t\le 2^{m}r\leq 8\,\varphi\!\bigl(2c\,\varphi(c\,d_\Omega(x))\bigr).
	\eeqq	
Thus the claim holds.	
	
Substituting this into \eqref{eq:last-start}, we conclude 
	\begin{equation*}\label{eq:last-case-3}
		\operatorname{diam}(\alpha[a,x])
		\le
		2\,\varphi\Bigl(16c\,\varphi\bigl(2c\,\varphi(c\,d_\Omega(x))\bigr)\Bigr).
	\end{equation*}

Hence \eqref{eq:2.1} holds in all cases. Since \eqref{eq:2.2} was already established, 
it follows that $\Omega$ is a $\psi$-diam uniform domain, where
$$\psi(t) = 2\varphi\Bigl(16c\,\varphi\bigl(2c\,\varphi(c\,t)\bigr)\Bigr).$$
\end{proof}

\begin{proof}[Proof of Theorem \textup{\ref{thm:the equivalence of the three types of varphi uniform domain} \eqref{result 2}}]
We first prove that  every $\varphi$-diam John domain is an $\eta$-length John domain, where $$\eta(t)=\max\{2\eta_1(2t), 2\eta_2(t)\},$$
and
$$\eta_1(t) = \int_0^{\varphi(Ct)} Q^{\log_2 160C\frac{s}{\varphi^{-1}(s)}} ds,\quad  \eta_2(t)=\left(Q^{\log_2 80\,\varphi\left(2Ct\right)/t}+C'\right)t,$$
provided $\eta_1$ converges and $C$, $C'$, satisfying $C\geq C'\geq1$, are constants depending only on the given data.

Since $\Omega$ is $\varphi$-diam John, then for each pair of points $x_1, x_2 \in \Omega$, there is a curve $\gamma\in \Omega$ joining $x_1$ and $x_2$ and $c \ge 1$, such that for each $x \in \gamma$, 
 $$\min\left\{\diam (\gamma[x_1, x]), \diam (\gamma[x_2, x])\right\} \leq \varphi\left(c\,d_\Omega(x)\right).  $$
Select $x_0 \in \gamma$ such that 
$$\diam (\gamma[x_1, x_0]) = \diam (\gamma[x_2, x_0]) \leq \varphi\left(c\,d_\Omega(x_0)\right).$$ 
Then for any $x \in \gamma[x_1, x_0]$ and $y \in \gamma[x_2, x_0]$,
\beq\label{eq:1.1}
\begin{aligned}
&\diam (\gamma[x_1, x]) = \min\left\{\diam (\gamma[x_1, x]), \diam (\gamma[x_2, x])\right\} \leq \varphi\left(c\,d_\Omega(x)\right), \\
&\diam (\gamma[x_2, y]) = \min\left\{\diam (\gamma[x_1, y]), \diam (\gamma[x_2, y])\right\} \leq \varphi\left(c\,d_\Omega(y)\right). 
\end{aligned}
\eeq
We now construct a curve $\gamma_0$ joining $x_1$ and $x_0$ such that for each $x \in \gamma_0$, 
\beq\label{eq:1.2}
\ell(\gamma_0[x_1, x])\leq\eta\left(d_{\Omega}(x)\right).
\eeq
Similarly, one constructs $\gamma'_0$ joining $x_2$ and $x_0$. 
The desired curve is then obtained by removing a loop from $\gamma_0\cup\gamma'_0$. 

Note that since $\varphi(t)$ is a increasing function and $ t/\varphi^{-1}(t) $ is decreasing when $0<t<1$ and increasing when $t\geq1$, so $\varphi(t)/t $ is decreasing when $0<t<\varphi^{-1}(1)$ and increasing when $t\geq\varphi^{-1}(1)$. 
%Assume without loss of generality that $\varphi^{-1}(1)=1$. 

We now split the proof into two cases according to the size of
$\operatorname{diam}(\gamma[x_1, x_0])$.

\textbf{Case 1:} $\operatorname{diam}(\gamma[x_1, x_0]) \le \varphi(10c\,\varphi^{-1}(1))$.

First, we construct a sequence of points along $\gamma[x_1, x_0]$.
Let $z_0 = x_0$. For $i \ge 0$, suppose $z_i$ has been defined. Let $z_{i+1} \in \gamma[x_1, z_i]$ be the last point along the direction from $z_i$ to $x_1$ such that
\[
d(z_{i+1}, z_i) = r_i := \frac{\varphi^{-1}\bigl(\operatorname{diam}(\gamma[x_1, z_i])\bigr)}{10c}.
\]
If for some $i_0 \in \mathbb{N}$ we have $z_{i_0} = x_1$, then we set $z_{i_0+k} = x_1$ for all $k \ge 0$. By \eqref{eq:1.1} and the assumption, we have
$d_\Omega(z_i) \ge 10r_i$ and $r_i < \varphi^{-1}(1)$
for all $i \in \mathbb{N}$.

By construction, $z_i, z_l \notin B(z_j, r_j)$ whenever $i < j < l$, and the balls $B(z_i, \frac{1}{2}r_i)$ are pairwise disjoint.

Since $X$ is a length space, for each $i \in \mathbb{N}$, there exists an arc $\alpha_i$ joining $z_i$ to $z_{i+1}$ such that
\begin{equation}\label{eq:length}
	\ell(\alpha_i) \le 2d(z_i, z_{i+1}) = 2r_i \le \frac{1}{5}d_\Omega(z_i).
\end{equation}
Hence $\alpha_i \subset \Omega$. Let
$\alpha' = \bigcup\limits_{i \in \mathbb{N}} \alpha_i.$
After deleting loops from $\alpha'$, we obtain a curve joining $x_1$ to $x_0$, denoted by $\gamma_0$. We now prove that this curve is the desired one.

Write $L=\diam (\gamma[x_1, x_0])$ and 
$$A_j =\{z_l:  2^{-j-1}L<\diam (\gamma[x_1, z_l])\leq 2^{-j}L, \, l\in\mathbb{N}\} , \quad j \geq 0.$$ 
Clearly if $z_l \in A_j$, since  $\varphi(t)\geq t$ and $\varphi$ is increasing, we have 
\beq\label{eq:1.4}
\frac{1}{10c}\varphi^{-1}\!\left(2^{-j-1} L\right)< r_l\leq \frac{1}{10c}\varphi^{-1}\!\left(2^{-j}L\right)\leq \frac{1}{10c}2^{-j}L,
\eeq
and by \eqref{eq:1.1},
\beq\label{eq:1.4-1}
d_{\Omega}(z_l)\geq \frac{1}{c}\varphi^{-1}\left(2^{-j-1}L\right).
\eeq
Then for $y\in B(z_l, r_l)$, by \eqref{eq:1.4},
$$d(x_1, y) \leq d(x_1, z_l) + d(z_l, y) \leq \diam(\gamma[x_1, z_l])  + r_l \leq 2^{-j+1}L,$$ 
so $$B\left(z_l,  \frac{1}{20c}\varphi^{-1}\left(2^{-j-1}L\right)\right)\subset B(z_l, r_l/2)\subset B(z_l, r_l)
\subset B\left(x_1, 2^{-j+1}L\right).$$
Since $ (X, d)$ is $Q$-doubling,  Lemma \ref{qs-5} implies $$\Card(A_j)\leq Q^{\left[\log_2 20c\, \frac{2^{-j+1}L}{\varphi^{-1}\left(2^{-j-1}L\right)}\right]}.$$

Fix $x \in \gamma_0$, then there exists $i_1\in \mathbb{N}$ such that $x \in \alpha_{i_1}$. Let $\alpha'_{i_1}=\bigcup\limits_{j=i_1}^{\infty}\alpha_i$ and $j_1$ be the least index $j$ such that
$z_{i_1}\in A_{j_1}$, by \eqref{eq:length} and \eqref{eq:1.4},  

\beq\label{eq:1.5}
\ell(\alpha'_{i_1}) \leq 2\sum_{j={i_1}}^{\infty} d(z_j, z_{j+1}) = 2\sum_{j=j_1}^{\infty} \sum_{\substack{z_l\in A_j \\  l \geq i_1}} r_l\leq \frac{1}{5c}\sum_{j=j_1}^{\infty} 2^{-j}L\,Q^{\log_2 20c\, \frac{2^{-j+1}L}{\varphi^{-1}\left(2^{-j-1}L\right)}}.
\eeq
Let 
$j_1'\geq0$ be the unique integer such that  $2^{-j-1}L\geq1$ when $j< j_1'$  and  $2^{-j-1}L<1$ when $j\geq j_1'$. 
Without loss of generality, we may assume that $j_1'> j_1\geq0$. 
Since the function $t/\varphi^{-1}(t)$ is decreasing when $0<t<1$ and increasing when $t\geq1$,   
we obtain that
\beq\label{eq:1.6}
\begin{aligned}
	&\quad\sum_{j=j_1}^{\infty} 2^{-j}L\,Q^{\log_2 20c\, \frac{2^{-j+1}L}{\varphi^{-1}\left(2^{-j-1}L\right)}}\\
	& =\sum_{j=j_1}^{j_1'-1} 2^{-j}L\,Q^{\log_2 20c\, \frac{2^{-j+1}L}{\varphi^{-1}\left(2^{-j-1}L\right)}}+ \sum_{j=j'_1}^{\infty} 2^{-j}L\,Q^{\log_2 20c\, \frac{2^{-j+1}L}{\varphi^{-1}\left(2^{-j-1}L\right)}}\\
	&  \leq C_1+ 4 \int_0^{2^{-j'_1-1}L} Q^{\log_2 80c\,\frac{t}{\varphi^{-1}(t)}}dt\\
	&\leq C_1+ 4 \int_0^{\varphi(c\,d_{\Omega}(z_{i_1}))} Q^{\log_2 80c\,\frac{t}{\varphi^{-1}(t)}}dt,	
\end{aligned}
\eeq
where $$C_1=\sum_{j=j_1}^{j_1'-1} 2^{-j}L\,Q^{\log_2 8\frac{ \varphi\left(10c\,\varphi^{-1}(1)\right)}{\varphi^{-1}(1)}},$$
 and the last inequality is due to \eqref{eq:1.4-1}.
Since $\varphi(t)\geq t$, by \eqref{eq:1.5} and \eqref{eq:1.6},
\beqq
\begin{aligned}
	\ell(\alpha'_{i_1})&\leq \frac{1}{5c}C_1+ \frac{4}{5c}  \int_0^{\varphi(c\,d_{\Omega}(z_{i_1}))} Q^{\log_2 80c\frac{t}{\varphi^{-1}(t)}}dt\\
	&\leq \int_0^{C_1+\varphi(c\,d_{\Omega}(z_{i_1}))} Q^{\log_2 80c\frac{t}{\varphi^{-1}(t)}}dt\\
	&\leq  \int_0^{\varphi(Cd_{\Omega}(z_{i_1}))} Q^{\log_2 80c\frac{t}{\varphi^{-1}(t)}} dt,
\end{aligned}
\eeqq
where $C\geq c$
%$$C\geq 8\varphi\!\left(10c\right) Q^{\log_2 8\,\frac{ \varphi\left(10c\,\varphi^{-1}(1)\right)}{\varphi^{-1}(1)}}+ 4\int_0^{1} Q^{\log_2 80c\,\frac{t}{\varphi^{-1}(t)}} dt$$
is a constant related to $c,Q$ and the doubling constant of $\varphi^{-1}$, and the last inequality holds based on \eqref{eq:1.4-1} and Lemma \ref{lem-3-1}.
Therefore $$\ell(\alpha'_{i_1})\leq \eta_1(d_{\Omega}(z_{i_1})),$$ 
where $$\eta_1(t) = \int_0^{\varphi(Ct)} Q^{\log_2 160C\frac{s}{\varphi^{-1}(s)}} ds,$$
under the standing assumption that the integral converges. 

On the other hand, for any $\varepsilon > 0$, we may pick $y \in \partial\Omega$ such that
\beq\label{eq:1.13}
d_{\Omega}(x) \geq d(y, x)-\varepsilon\geq d(y, z_{i_1}) - d(z_{i_1}, x) - \varepsilon \stackrel{\eqref{eq:length}}{\geq} d_{\Omega}(z_{i_1}) - 2r_{i_1} - \varepsilon \geq \frac{4}{5}d_{\Omega}(z_{i_1}) - \varepsilon,
\eeq
which implies that 
\beq\label{eq:boundary distance}
d_{\Omega}(z_{i_1}) \leq \frac{5}{4}d_{\Omega}(x). 
\eeq
Then
\beq\label{eq:1.7}
\ell(\gamma_0[x_1, x])\leq\ell(\alpha'[x_1, x])\leq\ell(\alpha'_{i_1})\leq \eta_1(d_{\Omega}(z_{i_1}))\leq\eta_1(2d_{\Omega}(x)).
\eeq 

In particular,  from \eqref{eq:1.5} and \eqref{eq:1.6}, we know
\beq\label{eq:1.8}
\begin{aligned}
	\ell(\alpha')\leq\ell(\alpha'_0)
	&\leq \varphi\!\left(10c\,\varphi^{-1}(1)\right) Q^{\log_2 8\,\frac{ \varphi\left(10c\,\varphi^{-1}(1)\right)}{\varphi^{-1}(1)}}+ \int_0^{1} Q^{\log_2 80c\,\frac{t}{\varphi^{-1}(t)}} dt
	\leq C_2,
\end{aligned}\eeq
where $C_2\geq1$ is a constant depended on the given data.

\textbf{Case 2:} $\diam(\gamma[x_1, x_0])> \varphi\left(10c\,\varphi^{-1}(1)\right)$. 

Let $y_1$ be the first point on $\gamma[x_1, x_0]$ such that 
$\diam(\gamma[x_1, y_1]) = \varphi\left(10c\,\varphi^{-1}(1)\right)$. 
On the arc $\gamma[x_1, y_1]$, 
%since $\diam(\gamma[x_1, y_1]) = \varphi(10c)$, 
we apply the construction in Case~1 to obtain a curve $\gamma_1$ 
joining $x_1$ and $y_1$.
For the remaining arc $\gamma[y_1, x_0]$, we proceed as follows. 
Set $w_0 = x_0$ and define inductively: 
for each $j \ge 0$, let
\[
r'_j = \frac{\varphi^{-1}\bigl(\diam(\gamma[x_1, w_j])\bigr)}{10c}.
\]
If $d(w_j, y_1) < r'_j$, set $p = j$ and stop. 
Otherwise, let $w_{j+1} \in \gamma[y_1, w_j]$ be the last point 
from $w_j$ toward $y_1$ with $d(w_{j+1}, w_j) = r'_j$, 
and continue. Finally, set $w_{p+1}=y_1$.

By \eqref{eq:1.1} and the assumption, we know 
$d_{\Omega}(w_j) \geq 10r'_j$ and $r'_j\geq\varphi^{-1}(1)$ for all $j \in \{0,\cdots,p\}$. Clearly, $w_i\notin B(w_j, r'_j)$ for $i\neq j$, and the balls
$B\left(w_j, \frac{1}{2} r'_j\right)$ are pairwise disjoint for $j\in\{0,\cdots,p\}$. 

Similarly, we also can find an arc $\beta_j\subset\Omega$ joining $w_j$ to $w_{j+1}$ such that
\begin{equation}\label{eq:length2}
	\ell(\beta_j) \le 2d(w_j, w_{j+1}) = 2r'_j \le \frac{1}{5}d_\Omega(w_j).
\end{equation}
Removing finitely many loops from $\gamma_1\cup\gamma_2$, where $\gamma_2 = \bigcup\limits_{j=0}^{p} \beta_i,$ yields a curve $\gamma_0$ joining $x_1$ and $x_0$. We now show that $\gamma_0$ is the desired curve.

Fix $x \in \gamma_0$. If $x \in \gamma_1$, by \eqref{eq:1.7}, we know $$\ell(\gamma_0[x_1,x])\leq\ell(\gamma_1[x_1,x])\leq\eta_1(2d_{\Omega}(x)).$$
It thus remains to handle the case $x \in \gamma_2$ for which there exists $i_2 \in \{0,\cdots,p\}$ so that $x \in \beta_i.$ 
 
If $i_2=p$, then from \eqref{eq:1.8} and \eqref{eq:length2},   
\beqq 
\ell(\gamma_0[x_1,x])\leq\ell(\gamma_1)+\ell(\gamma_2[y_1,x])\leq C_2+2r'_p\leq C_2+\frac{1}{5}d_{\Omega}(w_p)\leq \left(\frac{C_2}{\varphi^{-1}(1)}+\frac{1}{4}\right)d_{\Omega}(x),
\eeqq
where the last inequality can be obtained by a discussion similar to that of \eqref{eq:1.13} and \eqref{eq:boundary distance}.

If  $i_2\leq p-1$, then $p\geq1$. Set 
$$\kappa=Q^{\left[\log_2 8\varphi(10c\,r'_{i_2})/r'_{i_2}\right]}+1.$$ 
%where $c'\geq1$ is a constant, depending only on $c$ and the doubling constant of $\varphi^{-1}$.  
Claim that 
$r'_{i+k}\leq\frac{1}{2}r'_i$ 
whenever $i\geq i_2$, $k\geq\kappa$ and $i, \cdots,i+k\in\{1, \cdots,p\}$.

Suppose to the contrary that for any fixed $i\geq i_2$, 
$$r'_j > \frac{1}{2} r'_i\quad \text{for} \quad j = i+1, \cdots, i+k.$$  
We will show that $k$ is bounded.  
Indeed, for each $i \leq j \leq i + k$ and each $y \in B(w_j, r'_j/2)$, we have  
$$d(x_1, y) \leq d(x_1, w_j) + d(w_j, y) \leq \varphi(10c\,r'_j) + r'_j/2 \leq \varphi(10c\,r'_{i})+r'_{i}/2\leq2\varphi(10c\,r'_{i}).$$
%where $c'\geq1$ is a constant, depending only on $c$ and the doubling constant of $\varphi^{-1}$.
Thus 
$$B(w_j, r'_i/4) \subset B(w_j, r'_j/2) \subset B(x_1, 2\varphi(10c\,r'_{i})).$$ 
Since $ (X, d)$ is $Q$-doubling and $r'_i\geq\varphi^{-1}(1)$, Lemma \ref{qs-5} implies  
$$k \leq Q^{\left[\log_2 8\varphi(10c\,r'_{i})/r'_{i}\right]} \leq Q^{\left[\log_2 8\varphi(10c\,r'_{i_2})/r'_{i_2}\right]},$$
and so the claim holds for all $ k \geq \kappa = Q^{\left[\log_2 8\varphi(10c\,r'_{i_2})/r'_{i_2}\right]}+1$.

Following \eqref{eq:1.8} and \eqref{eq:length2}, we have
\beq\label{eq:1.11}
\begin{aligned}
\ell(\gamma_0[x_1, x])  &\leq \ell(\gamma_1)+2 \sum_{j=i_2}^{p} r'_j\leq C_2+2 \sum_{k=0}^{[\frac{p}{\kappa}]} \left( \sum_{j=k\kappa}^{(k+1)\kappa-1} r'_{i_2+j} \right)+2r'_p \\
&\leq C_2+2 \left(\sum_{k=0}^{\infty} \kappa\,  2^{-k}+1\right)r'_{i_2} \leq C_2+(\frac{2}{5}\kappa+\frac{1}{5}) \,d_{\Omega}(w_{i_2}) \\
&\stackrel{\eqref{eq:length2}}{\leq} \left(\frac{2}{5}\kappa+\frac{C_2}{\varphi^{-1}(1)}+\frac{1}{5}\right)d_{\Omega}(w_{i_2}).
\end{aligned} 
\eeq
In the sequel, we adopt the convention that $r'_{p+k}=0$ whenever $k\geq1$.

On the other hand, similar to \eqref{eq:1.13}, we have 
$$d_{\Omega}(w_{i_2}) \leq \frac{5}{4}d_{\Omega}(x) \quad \text{and} \quad \varphi^{-1}(1)\leq r'_{i_2}\leq\frac{1}{8}d_{\Omega}(x).$$
Therefore
$$\kappa = Q^{\left[\log_2 8\varphi(10c\,r'_{i_2})/r'_{i_2}\right]}+1\leq Q^{\left[\log_2 64\varphi\left(2c\,d_{\Omega}(x)\right)/d_{\Omega}(x)\right]}+1.$$
It follows from \eqref{eq:1.11} that
\[\ell(\gamma_0[x_1, x]) \leq \left(\frac{1}{2}Q^{\left[\log_2 64\,\varphi\left(2c\,d_{\Omega}(x)\right)/d_{\Omega}(x)\right]}+\frac{5}{4\,\varphi^{-1}(1)}C_2+\frac{3}{4}\right)d_{\Omega}(x) \leq \eta_2 (d_{\Omega}(x)),\]
where, up to a change of constants, 
$$\eta_2(t)=\left(Q^{\log_2 80\,\varphi\left(2Ct\right)/t}+C'\right)t,$$
and $C\geq C'\geq1$.

From the two cases considered above, we obtain the unified estimate by setting 
$$\eta(t)=\max\{2\eta_1(2t), 2\eta_2(t)\}.$$
With this choice of $\eta$, \eqref{eq:1.2} holds in both cases. One easily checks that $\eta$ is continuous and increasing, with $\eta(0) = 0$ and $ \eta(t) \geq t$ for all $t > 0$.
Therefore $\Omega$ is a $\eta$-length John domain.

In the following, we prove that  $\varphi$-diam uniform domain is $\eta$-length uniform domain. 

Based on the previous proof, we further assume that
\begin{equation}\label{eq:1.12}
	\operatorname{diam}(\gamma) \le \varphi\bigl(c\,d(x_1,x_2)\bigr).
\end{equation}
It suffices to prove that the curve $\gamma_0$ constructed above satisfies
$$\ell(\gamma_0) \le \frac{1}{2}\eta(d(x_1,x_2)).$$

If $\operatorname{diam}(\gamma[x_1, x_0]) \le \varphi(10c\,\varphi^{-1}(1))$, then from \eqref{eq:1.5}, \eqref{eq:1.6} and \eqref{eq:1.7}, we obtain
\begin{equation*}
\begin{aligned}
\ell(\gamma_0) &\leq C_1+ \int_0^{\diam(\gamma[x_1, x_0])} Q^{\log_2 80c\,\frac{t}{\varphi^{-1}(t)}} dt\\
 &\stackrel{\eqref{eq:1.12}}{\leq} C_1+ \int_0^{\varphi\left(c\,d(x_1,x_2)\right)} Q^{\log_2 80c\,\frac{t}{\varphi^{-1}(t)}} dt\\ 
&\stackrel{\eqref{eq:1.12}}{\leq}\eta_1(d(x_1,x_2))\leq \frac{1}{2}\eta(d(x_1,x_2)).
\end{aligned}
\end{equation*}

If instead $\operatorname{diam}(\gamma[x_1, x_0]) > \varphi(10c\,\varphi^{-1}(1))$, then by \eqref{eq:1.11},
\beq\label{1.14}
\ell(\gamma_0) \leq \ell(\gamma_1)+2 \sum_{j=0}^{p} r'_j 
\leq C_2+2r'_p+2 \sum_{k=0}^{\infty} \kappa\, 2^{-k}r'_0
 \leq C_2+(4\kappa+2) r'_0,
\eeq
where $$\varphi^{-1}(1)\leq r'_0=\frac{\varphi^{-1}\left(\diam\left(\gamma[x_1, w_0]\right)\right)}{10c}\leq \frac{\varphi^{-1}\left(\diam\,\gamma\right)}{10c}\stackrel{\eqref{eq:1.12}}{\leq}\frac{1}{10}d(x_1,x_2), $$
and $$\kappa = Q^{\left[\log_2 8\,\varphi(10c\,r'_{0})/r'_{0}\right]}+1\leq Q^{\left[\log_2 80\,\varphi\left(c\,d(x_1,x_2)\right)/d(x_1,x_2)\right]}+1. $$
From \eqref{1.14}, we obtain
\begin{equation}
	\begin{aligned}
		\ell(\gamma_0)
		&\leq \left(\frac{C_2}{\varphi^{-1}(1)}+4\kappa+2\right) r'_0\\
		&\leq \left(\frac{C_2}{10\,\varphi^{-1}(1)}+\frac{2}{5}Q^{\left[\log_2 80\,\varphi\left(c\,d(x_1,x_2)\right)/d(x_1,x_2)\right]}+\frac{3}{5}\right)d(x_1,x_2)\\
		&\leq \eta_2(d(x_1,x_2))\leq \frac{1}{2}\eta(d(x_1,x_2)).
	\end{aligned}
\end{equation} 

Therefore, $\Omega$ is an $\eta$-length uniform domain. 
The proof of Theorem~\ref{thm:the equivalence of the three types of varphi uniform domain} is now complete.
\end{proof}

\subsection{Proof of Corollary~\ref{cor:quasihyperbolic geodesic}}
This subsection is devoted to the proof of Corollary~\ref{cor:quasihyperbolic geodesic}, which follows from Theorem~\ref{thm:the equivalence of the three types of varphi uniform domain} \eqref{result 2}.
\begin{proof}[Proof of Corollary~\textup{\ref{cor:quasihyperbolic geodesic}}]
For each pair of points $x_1, x_2 \in \Omega$, there is a quasihyperbolic geodesic $\gamma$ joining $x_1$ and $x_2$.
If
 $$ d(z_1, z_{2}) \leq r_1=\frac{\varphi^{-1}\left(\diam\left(\gamma[x_1, z_1]\right)\right)}{10c}\leq\frac{1}{10}d_{\Omega}(z_1),$$
then Lemma~\ref{Lemma-2.1} gives
$$\ell(\gamma[z_1, z_{2}])\leq\frac{100}{81}e^{\frac{10}{81}}d(z_1,z_2) \le 2d(z_1, z_{2}).$$
The proof of Corollary~\ref{cor:quasihyperbolic geodesic} is  therefore completed by following the argument in Theorem~\ref{thm:the equivalence of the three types of varphi uniform domain} \eqref{result 2}.
 \end{proof}	
\subsection{Proofs of Theorems \ref{thm:main result varphi inner uniform} and \ref{thm:main result varphi uniform}}

%The proof of Theorem \ref{thm:main result varphi inner uniform} 
In this subsection, we will prove Theorem \ref{thm:main result varphi inner uniform} and Theorem \ref{thm:main result varphi uniform} with the aid of \cite[Theorem 1.9]{GHW}.

\begin{proof}[Proof of Theorem \ref{thm:main result varphi inner uniform}]
Since $\Omega$ satisfies $c_0$-ball separation condition, by \cite[Theorem 1.9]{GHW}, we know that $\Omega$ satisfies $c_1$-Gehring-Hayman inequality with $c_1=c_1(Q, c_0)$. Thus, for any
$x_1, x_2\in \Omega$ and $\gamma=\gamma_{x_1x_2}\in\Lambda_{x_1x_2}(\Omega)$, we have
\beq\label{220}
\ell(\gamma)\leq c_1 \sigma(x_1, x_2).
\eeq
It suffices to prove for each $x\in \gamma$, 
there is a constant $c_2\geq1$ such that
\beq\label{221} 
\min \left\{\ell(\gamma[x_1, x]), \ell( \gamma[x_2, x])\right\} \leq \eta(c_2\,d_{\Omega}(x)).
\eeq
Select $x_0 \in \gamma$ such that
\beq\label{222}
\min \left\{\diam( \gamma[x_1, x_0]), \diam( \gamma[x_2, x_0])\right\} \geq \frac{1}{3} \diam(\gamma).
\eeq

We claim  that for each $i \in \{1,2\}$ and $x\in  \gamma[x_i, x_0]$,
\be\label{H25-0729-1}
\diam( \gamma[x_i ,x])\leq 4c\, \varphi(c_0d_{\Omega}(x)).
\ee
Once this claim is proved, it follows that $\Omega$ is a $\varphi$-diam John domain. By Corollary~\ref{cor:quasihyperbolic geodesic}, $\gamma$ is an $\eta$-length John curve satisfying \eqref{221}, where $\eta$ is as defined therein. Thus it remains to prove the above claim. 

It suffices, by symmetry, to consider the case $x\in \gamma[x_1,x_0]$.
We will prove that
\beq\label{333}
\diam(\gamma[x_1, x]) \leq 4c\, \varphi(c_0d_{\Omega}(x)).
\eeq
Choose $u\in \gamma([x_1,x])\,\cap\,\Omega \backslash B \left(x,\frac{1}{4}\diam( \gamma[x_1, x])\right)$, 
and by \eqref{222}, one can find $v \in \gamma[x_2,x]\,\cap\,
\Omega \backslash B \left(x,\frac{1}{4}\diam(\gamma[x_1, x])\right)$. The $\varphi$-LC-2 condition provides a constant $c>0$ and a curve
$\beta\subset\Omega$ joining $u$ and $v$ such that
$$\beta \subset \Omega\backslash B\left(x,\varphi^{-1}(\frac{1}{4c}\diam( \gamma[x_1, x]))\right).$$
Applying the $c_0$-ball separation condition to $\gamma[u, v]$ and $\beta$, we obtain
\beqq
B_{\sigma}(x,c_0d_{\Omega}(x)) \cap \beta \neq \emptyset,
\eeqq
hence
\beqq
\diam(\gamma[x_1, x]) \leq 4c\,\varphi(c_0d_{\Omega}(x)).
\eeqq
This establishes \eqref{333} and  thereby completes the proof of \eqref{H25-0729-1}. The proof of Theorem~\ref{thm:main result varphi inner uniform} is therefore complete.
\end{proof}

\begin{proof}[Proof of Theorem \ref{thm:main result varphi uniform}]

	By the proof of Theorem \ref{thm:main result varphi inner uniform},  we know 
	%that  every 	quasihyperbolic geodesic $\gamma$ in $\Omega$ is $\varphi$-inner uniform curve, that is  
	for any $x, y\in \Omega$ and $\gamma_{xy}\in \Lambda_{xy}(\Omega)$, there exists a constant $c\geq1$ such that for all $z\in \gamma_{xy}$, it holds
	\beq\label{H25-0525-0}
	\min\{\ell(\gamma_{xy}[x,z]), \ell(\gamma_{xy}[y,z])\}\leq \eta(c\,d_\Omega(z)),
	\eeq
	and
	\beq\label{H25-0525-1}
	\ell(\gamma_{xy})\leq c\,\sigma(x,y).
	\eeq
	
	Next, we prove that 
	$$\ell(\gamma_{xy})\leq e^{2c\kappa C_0}d(x, y),$$ 
	where $C_0=C_0(Q, c_0)$ is a constant given by Lemma \ref{Lemma-2.5} and $$\kappa=\kappa(d(x,y))=2Q^{\log_2 15c\,\varphi(c'd(x,y))/\eta^{-1}(d(x,y))}+4,$$
	 with $C_4$ a relatively large constant related to the given data.
	
	Suppose on the contrary that there exists some $\gamma_0\in\Lambda_{xy}(\Omega)$  such that
	$$\ell(\gamma_0)> e^{2c\kappa C_0}d(x, y).$$
	Then
	\be\label{H25-0525-2}
	d(x,y)\ge\frac{1}{2}\max\{d_\Omega(x),d_\Omega(y)\}.\ee
	Otherwise, by Lemma \ref{Lemma-2.1},
	$$\ell(\gamma_0)\le\frac{20}{9}e^{\frac{10}{9}}d(x,y).$$
	This contradiction shows that \eqref{H25-0525-2} holds.
	
	%Thus by \eqref{(2.1)} and \eqref{H25-0525-2}, we have
	%$$k_\Omega(x,y)=\ell_k(\gamma_0)\ge \log \Big(1+\frac{\ell(\gamma_0)}{\min\{d_\Omega(x),d_\Omega(y)\}}\Big)>1.$$
Since \(\Omega\) is \(\varphi\)-LC-1, for \(x, y \in B(x, \frac{6}{5}d(x,y)) \cap \Omega\), there exist \(c' \ge 1\) and a curve
\[
\alpha \subset B(x, \varphi(c'd(x,y))) \cap \Omega
\]
joining \(x\) and \(y\).

\textbf{Step 1.} Construct the sequences $\{x_{0,i}\}$.

From \eqref{H25-0525-1}, for any $\gamma\in \Lambda_{xy}(\Omega)$,
\beq\label{H25-0525-3}
\ell(\gamma)\geq\sigma(x,y)\geq  \frac{1}{c}\ell(\gamma_0)> e^{c\kappa C_0}d(x,y).\eeq
Let \(x_0 = x\), and let \(x_1\) be the last point on \(\alpha\) along the direction from \(x\) to \(y\) such that, for some \(\gamma_1 \in \Lambda_{xx_1}(\Omega)\),
\[
\ell(\gamma_1) \le e^{cC_0} d(x,y).
\]
Based on \eqref{H25-0525-3}, we repeat this procedure to obtain \([\kappa]\) successive points \(x_i \in \alpha\). For each \(i \in \{1, \dots, [\kappa]\}\), \(x_i\) is the last point in \(\alpha[x_{i-1}, y]\) along the direction from \(x_{i-1}\) to \(y\) such that, for some \(\gamma_i \in \Lambda_{xx_i}(\Omega)\),
\begin{equation}\label{H25-0520-5}
	\ell(\gamma_i) \le e^{cC_0i} d(x,y),
\end{equation}
and consequently,
\begin{equation}\label{H25-0520-6}
	\ell(\gamma_i) > e^{cC_0(i-1)} d(x,y) \ge e^{cC_0} d(x,y), \quad i \in \{2, \dots, [\kappa]\}.
\end{equation}
For each \(i \in \{2, \dots, [\kappa]\}\), let \(x_{0,i} \in \gamma_i\) be the first point along the direction from \(x_i\) to \(x\) such that
\begin{equation}\label{H25-0520-7}
	\ell(\gamma_i[x_i, x_{0,i}]) = d(x,y).
\end{equation}

\textbf{Step 2.} Prove that there exist $x_{0,s}, x_{0,t}$ with $s + 1 < t$ such that $k_\Omega(x_{0,s}, x_{0,t}) < 1$.

From \eqref{H25-0520-6} and \eqref{H25-0520-7}, we obtain
\begin{equation}\label{H25-0520-11}
	\ell(\gamma_i[x, x_{0,i}]) = \ell(\gamma_i) - \ell(\gamma_i[x_i, x_{0,i}]) > (e^{cC_0}- 1) d(x,y) > \ell(\gamma_i[x_i, x_{0,i}]),
\end{equation}
which, together with \eqref{H25-0525-0}, shows that
\begin{equation}\label{H25-0520-9}
	d(x,y) = \min\{\ell(\gamma_i[x_i, x_{0,i}]), \ell(\gamma_i[x, x_{0,i}])\} \le \eta(c\,d_\Omega(x_{0,i})).
\end{equation}

For each \(i \in \{2, \dots, [\kappa]\}\), set
\[
B_{1,i} := B\left(x_{0,i}, \frac{1}{5c}\eta^{-1}(d(x,y))\right).
\]
Then \eqref{H25-0520-9} implies that
\[
\overline{B_{1,i}} \subset B\left(x_{0,i}, \frac{1}{2}d_\Omega(x_{0,i})\right) \subset \Omega.
\]
Since \(\alpha \subset B(x, \varphi(c'd(x,y))) \cap \Omega\), for each \(u \in \overline{B_{1,i}}\),
\[
\begin{aligned}
	d(x, u)
	&\le d(x, x_i) + d(x_i, x_{0,i}) + d(x_{0,i}, u) \\
	&\le \varphi(c'd(x,y)) + d(x,y) + \frac{1}{5c}\eta^{-1}(d(x,y)) \\
	&\le 3\varphi(c' d(x,y)),
\end{aligned}
\]
 where the last inequality follows from \(\varphi(t) \ge t\) and \(\eta(t) \ge t\) for $ t>0$, together with the fact that $\eta$ is increasing.
 %  and Lemma~\ref{lem-3-1}, \(c'' \ge 1\) is a constant related to \(c\) and the doubling constant of \(\varphi^{-1}\).

Let
\[
B_{0,0} := B(x, 3\varphi(c' d(x,y))).
\]
Then \(\overline{B_{1,i}} \subset B_{0,0}\) for all \(i\).
Since
\[
[\kappa] - 2 > 2Q^{\left[\log_2 15c\,\varphi(c' d(x,y))/\eta^{-1}(d(x,y))\right]},
\]
Lemma~\ref{qs-5}, applied with \(R = 3\varphi(c' d(x,y))\) and \(r = \frac{1}{5c}\eta^{-1}(d(x,y))\), implies that there exist integers \(s, t\) with
$s + 1 < t \in \{3, \dots, [\kappa]\}$
such that
\[
\overline{B_{1,s}} \cap \overline{B_{1,t}} \neq \varnothing.
\]
Consequently,
\[
d(x_{0,t}, x_{0,s}) \le \frac{2}{5c}\eta^{-1}(d(x,y)) \stackrel{\eqref{H25-0520-9}}{\le} \frac{2}{5} d_\Omega(x_{0,t}).
\]
Applying Lemma~\ref{Lemma-2.1} with \(a = \frac{5}{2}\) gives
\[
k_\Omega(x_{0,s}, x_{0,t}) \le \frac{20}{27} < 1.
\]

\textbf{Step 3.} Derive a contradiction and conclusion.

For every $i\geq2$, by \eqref{(2.1)}, \eqref{H25-0525-2}, and \eqref{H25-0520-11}, we have
\[
\begin{aligned}
	k_\Omega(x, x_{0,i})
	&\ge \log\left(1 + \frac{\ell(\gamma_i[x, x_{0,i}])}{\min\{d_\Omega(x), d_\Omega(x_{0,i})\}}\right) \\
	&> \log\left(1 + \frac{(e^{cC_0} - 1)d(x,y)}{\min\{d_\Omega(x), d_\Omega(x_{0,i})\}}\right) \\
	&\ge \log\left(1 + \frac{e^{cC_0} - 1}{2}\right) > 1.
\end{aligned}
\]
Thus by Lemma~\ref{Lemma-2.5},
\[
\ell(\gamma_t[x, x_{0,t}]) \le C_0\, \ell(\gamma_s[x, x_{0,s}]).
\]
It follows from the above estimate and \eqref{H25-0520-7} that
\[
\begin{aligned}
	\ell(\gamma_t)
	= \ell(\gamma_t[x, x_{0,t}]) + \ell(\gamma_t[x_t, x_{0,t}])
	= \ell(\gamma_t[x, x_{0,t}]) + \ell(\gamma_s[x_s, x_{0,s}])\leq C_0\, \ell(\gamma_s).
\end{aligned}
\]
On the other hand, since \(s < t-1\), \eqref{H25-0520-5} and \eqref{H25-0520-6} imply
\[
\ell(\gamma_t) > e^{cC_0(t-1)} d(x,y) \ge e^{cC_0(s+1)} d(x,y) \ge e^{cC_0} \ell(\gamma_s).
\]
The above two estimates clearly contradict with each other. Consequently, it holds
\[
\ell(\gamma_{xy}) \le e^{2c\kappa C_0}d(x,y).
\]
This, combining with \eqref{H25-0525-0}, show the Theorem~\ref{thm:main result varphi uniform} holds.
\end{proof}

\medskip

\textbf{Acknowledgment.}  We would like to thank our supervisor Professor Chang-Yu Guo for helpful discussions and continuous guidance throughout this work.
%We would like to thank Prof.~Pekka Koskela,  Dr.~Abhishek Pandey, Yi Xuan, Zhiqiang Yang and Linhong Wang for numerous helpful discussions and comments. 
%We also thank Dr.~Abhishek Pandey and many other colleagues for their valuable comments and suggestions.
%We also thank Dr.~Abhishek Pandey for his kind comments about Theorem \ref{Thm-A} and Theorem \ref{thm:Zhou problem}, and many other colleagues for their valuable comments and suggestions.

\end{document}